\documentclass[12pt,twoside]{article}
\usepackage{a4}
\usepackage{amsfonts,amssymb,amscd,amsmath,enumerate,verbatim,calc,
  longtable,multirow, float, graphicx,breqn,bigstrut} 
\usepackage[utf8]{inputenc}
\usepackage[english]{babel}
\usepackage{mathrsfs}
\usepackage{mathtools,array,amsthm}
\usepackage{tabularx, tabulary}
\usepackage{hyperref}
\restylefloat{table}
\newtheorem{theorem}{Theorem}[section]

\newtheorem{corollary}[theorem] {Corollary}

\newtheorem{lemma} [theorem]{Lemma}

\newtheorem{proposition}[theorem]{Proposition}
\newtheorem{remark}[theorem]{Remark}

\newcommand{\Z}{\mathbb{Z}}

\newcommand{\ds}{\displaystyle}
\allowdisplaybreaks 

\newsavebox{\mybox}
\title{Trace and discriminant criteria for a matrix to be a sum of sixth and eighth powers of matrices}
\author{Rakesh Barai \& Anuradha S. Garge \\
\small 
Guru Nanak Khalsa College, \\
\small
Mumbai-400019, India. \\ \small \&\\
\small
Department of Mathematics, \\
\small 
University of Mumbai, Mumbai- 400098, India.} 
\date{}
\begin{document}
	\maketitle
	
	\begin{abstract}  
	  {\footnotesize  In this paper, we shall be considering the Waring's problem for matrices. One version of the problem
            involves writing an $n \times n$ matrix over a commutative ring $R$ with unity as a sum of $k$-th powers of
            matrices over $R.$ This study is motivated by the interesting results of Carlitz, Newman, Vaserstein, Griffin,
            Krusemeyer, Richman etc. obtained earlier in this direction. The results are for the case
            $n \geq k \geq 2$ in terms of the trace of the matrix.

            For $n < k,$ it was shown by Katre, Garge that it is enough to work with the special case $n = 2$ and $k \geq 3.$
            The cases $3 \leq k \leq 5$ and $k = 7$ were settled in earlier results. There was no case of a composite,
            non-prime-power $k$ occuring above. In this paper, we will find the trace criteria for a matrix to be a sum of sixth
            (a composite non-prime power) and eighth powers of matrices over a commutative ring $R$ with unity.

            An elegant discriminant criterion was obtained by Katre and Khule earlier in the special case
            of an order in an algebraic number field $\mathcal{O}.$ We will derive here similar discriminant criteria for
            every matrix over $\mathcal{O}$ to be a sum of sixth and eighth powers of matrices over $\mathcal{O}.$ 
		}
	\end{abstract}
 \section{Introduction}
 Let $R$ be a commutative ring with unity. Let $M \in M_n(R)$ and $k \ge 1$ be an integer.
 We say that $M$ is a sum of $k^{th}$ powers of matrices over $R$ if there exist $M_i\in M_n(R)$ such that $\ds M=\sum_{i=1}^r M_i^k$.
 All the matrices in $M_n(R)$ do not have this property. So one needs to determine:
\begin{enumerate}
	\item necessary and sufficient conditions for a matrix to be a sum of $k^{th}$ powers of matrices and 
	\item the least value of $r,$ the number of matrices required to express $M$ as the sum of $k^{th}$ powers of matrices. 
\end{enumerate}
This is known as the classical Waring's problem for matrices after the name of British mathematician Edward Waring who was first to
describe such a problem for integers. This problem was studied by L. Carlitz, M. Newman over the integers, by M. Griffin and M. Krusemeyer
over fields. L. N. Vaserstein, D. R. Richman studied this problem over commutative rings and S. A. Katre
gave a remarkable discriminant criterion for an order $R$ in an algebraic number field.

The majority of these had given the criteria (and the number of matrices requires) for a matrix of size 2, to be a sum of squares of the matrices. M. Griffin and M. Krusemeyer had given the criteria for a matrix of size 2, 3 and 4 to be a sum of squares
(Refer \cite{gk}, Propositions 3.3,  3.4 and 3.5). It was L. N. Vaserstein, who first gave a trace criterion for a matrix of any size to be a sum of squares of matrices. We recall his theorem below. 
\begin{theorem}[\cite{lv}, Theorem 1]
  Let $R$ be a ring with unity. Then for any integer $n\ge 2$, a matrix $M \in M_n(R)$ is a sum of squares if and only if
  Tr\footnote{Tr $M$ denotes the trace of the matrix $M$.} $M$ is a sum of squares in the ring $R/2R$.
\end{theorem}

It was D. R. Richman who first gave the trace criteria for a matrix of any size to be a sum of $k$-th powers of matrices,
provided $n\ge k\ge 2,$ in the form of the following theorem.
\begin{theorem}(\cite{dr}, Proposition 4.2) \label{th_dr}
	Let $R$ be a ring. If $n\ge k\ge 2$ are integers, then the following statements are equivalent:
	\begin{enumerate}
		\item $M$ is a sum of $k^{th}$ powers in $M_n(R).$
		\item $M$ is a sum seven $k^{th}$ powers in $M_n(R).$
		\item $M\in M_n(R)$ and for every prime power  $p^e$ dividing $k$, there exist elements
                  $x_0,x_1,\ldots, x_e$ (depending on $p$) in $R$ such that \[\mbox{Tr } M=x_0^{p^e}+px_1^{p^e-1}+\cdots + p^ex_e.\]
		  Moreover if $k=p$ is a prime, in statement (2) "seven" is replaced by "five" and statement (3) simplifies to
                  Tr $M=x_0^p+px_1,$ for some $x_0, x_1\in R.$
	\end{enumerate}
\end{theorem}

Let $K$ be an algebraic number field of degree $n.$ A subring of $K$ with unity which is also a $\Z$-module of $K$ of rank $n$ is called an order of $K.$ Let $\mathcal{O}$ be an order of $K.$ The  following theorem due to S. A. Katre and S. A. Khule gives
the discriminant criterion for a matrix to be a sum of $k^{th}$ powers of matrices over $\mathcal{O}.$ 
\begin{theorem}[\cite{kk}, Theorem 1]\label{th_kk}
	Let $\mathcal{O}$ be an order of a number field $K$ and $n\ge k\ge 2$. Then every matrix $M\in M_n(\mathcal{O})$ is  a sum of $k^{th}$ powers of matrices if and only if $(k,\mbox{ disc }\mathcal{O})=1.$	
\end{theorem} 
The restriction $n \ge k$ in Richman's theorem was removed by generalized trace criteria given by S. A. Katre and the second author 
via the following theorem: 
\begin{theorem}[\cite{kg}, Theorem 3.1] \label{th_kg}
	Let $R$ be a ring and $n,k\ge 2$ be integers. Let $M\in M_n(R)$. Then the following are equivalent:
	\begin{enumerate}
		\item $M$ is  a sum of $k^{th}$ powers of matrices in $M_n(R).$
		\item Tr $M$ is a sum of traces of $k^{th}$ powers of matrices in $M_n(R).$
		\item Tr $M$ is in the subgroup of $R$ generated by the traces of $k^{th}$ powers of matrices in $M_n(R).$
		\item Tr $M$ is in the subgroup of $R$ generated by the traces of $k^{th}$ powers of matrices in $M_n(R)$ (mod $k!R$).
		\item Tr $M$ is a sum of traces of $k^{th}$ powers of matrices in $M_n(R)\pmod{k!R}.$
	\end{enumerate}
\end{theorem}  

In the same paper, the proof of the following theorem is given:
\begin{theorem}[\cite{kg}, Theorem 3.5] \label{th_mn} Let $R$ be a ring and $n\ge m \ge 1$ and $k\ge 2$ be integers.
  If every matrix in $M_m(R)$ is a sum of $k^{th}$ powers of matrices in $M_m(R)$, then every matrix in $M_n(R)$ is
  also a sum of $k^{th}$ powers of matrices in $M_n(R).$
\end{theorem}

\begin{remark}\label{rem_thm3.5}
	We will use this theorem in the next section in the following way: 
        if we want to show every matrix in $M_n(R)$, $n>2$ satisfying a certain property is a sum of $k^{th}$ powers of matrices,
        then it is sufficient to prove it for every matrix in $M_2(R)$ satisfying the same property to be a sum of $k^{th}$
        powers of matrices in $M_2(R)$.
\end{remark}

The motivation for this paper is the following two theorems due to S. A. Katre and the second author.
\begin{theorem}[Trace criteria]\label{th_tr_23457}
  Let $R$ be a commutative ring with unity. Let $k=2,3,5,7.$ Then a matrix $M\in M_n(R)$ is a sum of
  $k^{th}$ powers of matrices over $R$ if the Tr $M$ is a $k^{th}$ power of an element in $R$ modulo $kR.$
	
	For $k=4,$ the matrix $M$ is a sum of fourth powers of matrices if the Tr $M=x_0^4+2x_1^2+4x_2,$ for some $x_0,x_1,x_2\in R.$
\end{theorem}

\begin{theorem}[Discriminant criteria] \label{th_dis_23457}
	Let $\mathcal{O}$ be an order of a number field and $k=2,3,5,7.$ Then every matrix over $\mathcal{O}$ is a sum of $k^{th}$ powers of matrices over $\mathcal{O}$  if and only if $(k,\mbox{disc }\mathcal{O})=1.$
	
	For $k=4$, the discriminant criterion  becomes $(2,\mbox{disc }\mathcal{O})=1.$
\end{theorem}

In this paper we will extend the results of the above type to sixth and eighth powers. Note here
that the case of sixth powers is the first non-prime power, composite number case under consideration. 
In particular, we shall derive the trace and discriminant criteria for a matrix $M$ to be a sum of sixth (and eighth)
powers of matrices over a commutative ring $R$ with unity and over an order $\mathcal{O},$ of an algebraic number field.
The results obtained are as follows: 

\begin{theorem} 
  Let $R$ be a commutative ring with unity and let $n \geq 2$ be an integer. 
  \begin{itemize}
  \item[(a)] A matrix $M \in M_2(R)$ is a sum of sixth powers in $M_2(R)$ if and only if Tr $M \pmod{6R} \in S,$ where
    $S=\left\lbrace x_0^6-2x_1^3+3x_2^2\pmod{6R}\;:\; x_0,x_1,x_2\in R\right\rbrace.$
    If $M \in M_n(R)$ is such that Tr $M \pmod{6R} \in S,$ then $M$ is a sum of sixth powers
    in $M_n(R).$ 

    Further, if $\mathcal{O}$ is an order in an algebraic number field, every matrix in $M_n(\mathcal{O})$
    is a sum of sixth powers of matrices if and only if $(6, {\rm~disc}(\mathcal{O})) = 1.$ 
    
    \item[(b)] A matrix $M \in M_n(R)$ is a sum of eighth powers in $M_n(R)$ if and only if Tr $M \pmod{8R} \in S_1,$ where
      $S_1 =\left\lbrace x_0^8 + 2x_1^4 + 4x_2^2 \pmod{8R}\;:\; x_0,x_1,x_2\in R\right\rbrace.$
      Further, if $\mathcal{O}$ is an order in an algebraic number field, every matrix in $M_n(\mathcal{O})$
    is a sum of eighth powers of matrices if and only if $(2, {\rm~disc}(\mathcal{O})) = 1.$ 
  \end{itemize} 
\end{theorem} 

We would like to mention here that recently S. A. Katre, Wadikar K. G. and Deepa Krishnamurthi (see \cite{kshk} and
\cite{kd}) have obtained various results for matrices over non-commutative rings too as sums of powers of matrices.

\section{Formula for trace of $A^n; A\in M_2(R)$}
Let $R$ be a commutative ring with unity and let $A=\begin{pmatrix} a & b \\ c & d \end{pmatrix} \in M_2(R).$
The following theorem due to Mc Laughlin J, gives the general formula for $A^n,$ for any positive integer $n.$ 
\begin{theorem}[\cite{la}, Theorem 1]\label{thm_A_raise_n}
  Let $R$ be a commutative ring with unity. Let $A=\begin{pmatrix} a & b \\ c & d \end{pmatrix} \in M_2(R).$
  Let $t =a+d$ and $\delta=ad-bc$ be the trace and determinant of $A$ respectively. Define
  $\displaystyle y_n=\sum_{r=0}^{\lfloor n/2\rfloor} \binom{n-r}{r} t^{n-2r}(-\delta)^r,$ where $\lfloor x\rfloor$ denotes the greatest integer $z$ such that $z\le x.$ Then for any positive integer
  \nolinebreak $n,$ 
	\[ A^n =\begin{pmatrix} y_n-dy_{n-1} & by_{n-1} \\ cy_{n-1} & y_n-ay_{n-1} \end{pmatrix}.\]
\end{theorem}
Based on this we have the following corollary.
\begin{corollary} \label{cor_tr}
	Using notation defined as above, \[\text{Tr } A^n=t^n+\sum_{r=1}^{[n/2]} (-1)^r \frac{n}{r}\binom{n-r-1}{r-1} t^{n-2r}\delta^r.\]
\end{corollary}
\begin{proof} From Theorem \eqref{thm_A_raise_n},
	\begin{align}
	\text{Tr }A^n=\ & 2y_n-(a+d)y_{n-1} \nonumber \\
	=& \sum_{r=0}^{\lfloor n/2\rfloor} 2\binom{n-r}{r}t^{n-2r}(-\delta)^r -t  \sum_{r=0}^{\lfloor (n-1)/2\rfloor} \binom{n-1-r}{r}t^{n-1-2r}(-\delta)^r, \nonumber \\
	=\ & t^n+ \left(\sum_{r=1}^{\lfloor n/2\rfloor} 2\binom{n-r}{r}-\sum_{r=1}^{\lfloor (n-1)/2\rfloor} \binom{n-1-r}{r}\right) t^{n-2r}(-\delta)^r \label{eq_trAn}
	\end{align}
	Let us assume $n=2l+1$ is odd. Then $\lfloor n/2\rfloor =l = \lfloor (n-1)/2\rfloor$. In this case,
	\begin{align}
	\mbox{Tr }A^n=\ & t^n+ \sum_{r=1}^{\lfloor n/2\rfloor} \left( 2\binom{n-r}{r}-\binom{n-1-r}{r}\right) t^{n-2r}(-\delta)^r \nonumber 
	\end{align}	
	Consider,
        $$2\binom{n-r}{r}- \binom{n-1-r}{r} = 2\frac{(n-r)!}{r!(n-2r)!}-\frac{(n-r-1)!}{r!(n-1-2r)!}$$ 
	$$= \frac{(n-r-1)!}{r!(n-1-2r)!}\left( 2\frac{n-r}{n-2r}-1\right) =  n \frac{(n-r-1)!}{r!(n-2r)!},\nonumber $$
	$$ = \frac{n}{r} \frac{(n-r-1)!}{(r-1)!(n-2r)!}	= \frac{n}{r} \binom{n-1-r}{r-1}.$$ 
	Substituting back in Equation \eqref{eq_trAn} we get desired expression.
	
	Now assume $n =2l$ be even number. Then $\lfloor n/2\rfloor=l$ and $\lfloor (n-1)/2\rfloor =l-1.$ Substituting in Equation \eqref{eq_trAn} we get,
	\begin{align*}
	\mbox{Tr }A^n =\ & t^n + \left(\sum_{r=1}^{l-1} 2\binom{n-r}{r}-\sum_{r=1}^{l-1} \binom{n-1-r}{r}\right) t^{n-2r}(-\delta)^r +2\binom{n-l}{l}(-\delta)^l \\
	=\ & t^n +\sum_{r=1}^{l-1}\frac{n}{r} \binom{n-1-r}{r-1}t^{n-2r}(-\delta)^r +\frac{n}{l} \binom{n-1-l}{l-1}(-\delta)^l \\
        =\ & t^n +\sum_{r=1}^{\lfloor n/2\rfloor}\frac{n}{r} \binom{n-1-r}{r-1}t^{n-2r}(-\delta)^r.
	\end{align*}
\end{proof}

\section{Trace Criterion for a Matrix to be a Sum of Sixth Powers}

We prove in this section, trace related results regarding sixth powers of matrices.
This is the first composite, non-prime power under consideration. We begin with a lemma. 

\begin{lemma}
  Let $M \in M_2(R).$ Then, $M = B^{6},$ for some $B\in M_2(R)$ if and only if there exist
  $t,\delta \in R$ such that Tr $M = t^6-6t^4\delta+9t^2\delta^2-2\delta^3.$
\end{lemma}
\begin{proof}
  Let $M \in M_2(R)$ be such that $M = B^6$ for some $B\in M_2(R)$. If $t=\mbox{Tr }B$ and $\delta=\mbox{det }B,$
  by Corollary \ref{cor_tr}, we get that Tr $M = \text{Tr} B^6 = t^{6} - 6t^4\delta + 3(3)t^{4}\delta^2 - 2\delta^3$ 
  $= t^6-6t^4\delta+9t^2\delta^2-2\delta^3.$

  Conversely if $M\in M_2(R)$ is such that Tr $M= t^6-6t^4\delta+9t^2\delta^2-2\delta^3$ with 
  $t, \delta \in R$, defining $B= \begin{pmatrix} t & \delta \\ -1 & 0\end{pmatrix} $, by Corollary \ref{cor_tr}
  we immediately get that Tr $M=$ Tr $B^6=t^6-6t^4\delta+9t^2\delta^2-2\delta^3.$ 
\end{proof}

Motivated from the above lemma, given a ring $R,$ we shall consider the set of traces of sixth powers modulo $6R$
and in fact prove that it is a subgroup of $R.$ 

\begin{lemma}\label{lem_S_subgp}
  Let $R$ be a commutative ring with unity. Define the set
  $$S=\left\lbrace x_0^6-2x_1^3+3x_2^2\pmod{6R}\;:\; x_0,x_1,x_2\in R\right\rbrace.$$
  Then $S$ is a subgroup of $R.$ 
\end{lemma}
\begin{proof}
  It is enough to  prove  $S$ is closed under addition and every element of $S$ has an additive inverse.

  In order to prove that $S$ is closed under addition, we note the identities: 
        \begin{description}
	\item[(i)] $(x+y)^3 = x^3+3x^2y+3yx^2+y^3$. Hence $-2(x+y)^3 \equiv -2x^3-2y^3\pmod{6R}$.
	\item[(ii)] $(z+w)^2=z^2+2zw+w^2.$ Hence $3(z+w)^2\equiv 3z^2+3w^2\pmod{6R}.$
	\item[(iii)] Using the above two identities, we see that
          $\ds (s+t)^6= s^6+ 6s^5t + 15s^4t^2 + 20s^3t^3 + 15s^2t^4 + 6st^5+ t^6 
	\equiv s^6+ t^6+ 3(st^2+ s^2t)^2+ 2s^3t^3 \pmod{6R}.$
\end{description}

We now prove that $S$ is closed under addition as follows: let $\alpha, \beta \in S.$
  Hence, there exist $x_0, x_1, x_2, y_0, y_1, y_2 \in R$ such that
  $\alpha = x_0^6-2x_1^3+3x_2^2\pmod{6R}$ and $\beta =  y_0^6- 2y_1^3 + 3y_2^2\pmod{6R}.$
  Hence using the above relations we get that
  $$\alpha + \beta \pmod{6R} \equiv (x_0^6-2x_1^3+3x_2^2) + (y_0^6- 2y_1^3 + 3y_2^2) \pmod{6R}$$ 
  $$\equiv (x_0+y_0)^6 - 2(x_0y_0)^3 + 3((x_0^2y_0)^2 + (x_0y_0^2)^2) -2(x_1+y_1)^{3} + 3(x_2 + y_2)^2 \pmod{6R}$$ 
   $$\equiv (x_0+y_0)^6 - 2(x_0y_0 + x_1+ y_1)^3 + 3(x_0^2y_0 + x_0y_0^2 + x_2 + y_2)^2 \pmod{6R}.$$ 
  Choose $z_0 = x_0 + y_0, z_1 = x_0y_0 + x_1 + y_1, z_2 = x_0^2y_0 + x_0y_0^2 + x_2 + y_2,$ to get that
  $\alpha + \beta = z_0^6- 2z_1^3 + 3z_2^2\pmod{6R}.$ This completes the proof that $S$ is closed
  under addition. 

  Note now that for any $a \in S,\ 2a,3a, \ldots, 6a \in S$. As $6a=0 $ in $S.$ So $a+5a=0$ in $S$ that is,
  $-a = 5a$ in $S$. Hence every element of $S$ has an additive inverse. Therefore $S$ is a subgroup of $R.$
\end{proof}

\begin{theorem}\label{thm_tr6}
  Let $R$ be a commutative ring with unity. Then $M\in M_2(R)$ is a sum of sixth powers of matrices in $M_2(R)$
  if and only if Tr $M \pmod{6R} \in S.$
\end{theorem}

\begin{proof}
  Let $M = \sum_{i=1}^{k} M_i^6,$ where $M, M_i \in M_2(R)$ for all $1 \leq i \leq k.$ Then,
  Tr $M \pmod{6R} \equiv \sum_{i=1}^{k} \text{Tr}~M_i^{6} \pmod{6R} = \sum_{i=1}^{k} (x_{i0}^6 - 2 x_{i1}^3 + 3x_{i2}^2) \pmod{6R}.$
  Since each term of this summation is in $S$ and $S$ is additively closed, from Lemma \ref{lem_S_subgp},
  we get that Tr $M \pmod{6R} \in S.$ 

  Assume $M \in M_2(R)$ with Tr $M \pmod{6R} \in S.$ 
  Therefore, 
	\begin{eqnarray}
	\text{Tr }M&=&x^6-2y^3+3z^2+6b'\ \text{ for some } b' \in R; \nonumber\\
	&=&x^6-2y^3+(2z^3+3z^2-6z+2)-2z^3-2+6z+6b', \nonumber \\
	&=&\text{Tr }\begin{pmatrix} x & 0 \\ 0 & 0 \end{pmatrix}^6+ \text{Tr }\begin{pmatrix} 0 & y \\ -1 & 0 \end{pmatrix}^6 + \text{Tr }\begin{pmatrix} 1 & -z+1 \\ -1 & 0 \end{pmatrix}^6 + \text{Tr }\begin{pmatrix} 0 & z \\ -1 & 0 \end{pmatrix}^6\nonumber\\ && \ \ \ \ +\text{Tr }\begin{pmatrix} 0 & 1 \\ -1 & 0 \end{pmatrix}^6 + 6(z+b').\nonumber
\end{eqnarray}

        Hence in order to prove that Tr $M$ is a sum of traces of sixth
        powers of matrices it is enough to prove that every element of the
        type $6b,$ $b \in R$  is a sum of traces of sixth powers of matrices. So consider,
	\begin{eqnarray}
	6b &=& (2b^3+9b^2+6b+1)-2b^3-9b^2-1, \nonumber \\
	&=&(2b^3+9b^2+6b+1)-2b^3+(-12b^2+3b^2)-1, \nonumber \\
	&=&\text{Tr}\begin{pmatrix} 1 & -b \\ -1 & 0 \end{pmatrix}^6 +\text{Tr}\begin{pmatrix} 0 & b \\ -1 & 0 \end{pmatrix}^6 +
        2(2b^6+3b^4-6b^2+2)  \nonumber  \\ &&  +(2b^6-6b^4+3b^2+2)-6b^6-6-1,\nonumber
	\\
	&=&\text{Tr}\begin{pmatrix} 1 & -b \\ -1 & 0 \end{pmatrix}^6 +\text{Tr}\begin{pmatrix} 0 & b \\ -1 & 0 \end{pmatrix}^6 + 2\text{Tr}\begin{pmatrix} -1 & b^2-1 \\ 1 & 0 \end{pmatrix}^6 +\text{Tr}\begin{pmatrix} -b & b^2-1 \\ -1 & 0 \end{pmatrix}^6  \nonumber \\ && +3\text{Tr}\begin{pmatrix} 0 & b^2 \\ -1 & 0 \end{pmatrix}^6 -8+1, \nonumber\\
	&=&\text{Tr}\begin{pmatrix} 1 & -b \\ -1 & 0 \end{pmatrix}^6 +\text{Tr}\begin{pmatrix} 0 & b \\ -1 & 0 \end{pmatrix}^6 + 2\text{Tr }\begin{pmatrix} -1 & b^2-1 \\ 1 & 0 \end{pmatrix}^6 +\text{Tr}\begin{pmatrix} -b & b^2-1 \\ -1 & 0 \end{pmatrix}^6\nonumber \\ && +3\text{Tr}\begin{pmatrix} 0 & b^2 \\ -1 & 0 \end{pmatrix}^6+4\text{Tr}\begin{pmatrix} 0 & 1 \\ -1 & 0 \end{pmatrix}^6+\text{Tr}\begin{pmatrix} 1 & 0 \\ -1 & 0 \end{pmatrix}^6.\nonumber
	\end{eqnarray}

        Applying the above with $b$ replaced by $z+b',$ we get that Tr $M$ is a sum of traces of sixth powers of matrices.
        By the equivalent conditions of Theorem \eqref{th_kg} (see $(2)$ implies $(1)$),
        we get that $M$ is a sum of sixth powers of matrices.
\end{proof}

\begin{corollary}\label{cor_equaltwoenough} 
	Let $M\in M_n(R),$ $n \geq 2.$ If Tr $M \pmod{6R} \in S,$ then $M$ is a sum of sixth powers of matrices in $M_n(R).$ 
\end{corollary}
\begin{proof}
  Let $M\in M_n(R)$ such that Tr $M \pmod{6R} \in S.$ Then there exist $x,y,z,b \in R$ such that Tr $M=x^6-2y^3+3z^2+6b.$
  By the theorem above, there exist $M_i\in M_2(R)$ such that $\sum_{i=1}^{r} \mbox{Tr }M_i^6 = x^6-2y^3+3x^2+6b.$ 
	For each $i$, define $N_i=\begin{pmatrix} M_i & 0 \\ 	0 & \ \ \ \ 0_{n-2}
	\end{pmatrix} \in M_n(R).$ Then Tr $N_i^6=\mbox{Tr }M_i^6$ and Tr $M= \sum_{i=1}^{r} \mbox{Tr }N_i^6.$
	Thus we get Tr $M$ is a sum of the traces of sixth powers of matrices over $R.$ Hence, $M$ is a sum of sixth powers
        in $M_n(R)$ by Theorem (\ref{th_kg}). 
\end{proof}

\section{Trace Criterion for a Matrix to be a Sum of Eighth Powers of Matrices}

In order to prove Theorem \eqref{th_dr}, Richman had used the theory of Witt vectors.
For any prime $p$ and integer $s\ge 1$, he had defined the set
$\displaystyle W(p,s,R)=\left\lbrace x_0^{p^s}+px_1^{p^{s-1}}+\cdots +p^sx_s\;:\; x_0,x_1,\ldots,x_s\in R\right\rbrace$
and proved the following results. 

\begin{proposition}[\cite{dr}, Proposition 3.1]\label{prop_powerps}
	$W(p,s,R)$ is closed under addition and subtraction.
\end{proposition}

\begin{proposition}[\cite{dr}, Proposition 3.2]\label{prop_traceps} 
	Let $M\in M_n(R)$ and $n \geq 2.$  Then for any prime $p$ and integer $s\ge 1$, Tr $\ds M^{p^s} \in W(p,s,R).$
\end{proposition}

We now concentrate on the special case of our interest: $p = 2, s = 3.$ Note that by the above proposition,
given $M \in M_n(R),$ there exist $x_0, x_1, x_2, x_3 \in R$ such that Tr $\ds M^{2^3} =$  Tr $M^{8} = x_0^{8}+ 2x_1^{4}+ 4x_2^2 + 8x_3.$
Let us define $\ds S_1 = W(2,3,R)({\rm~mod}8R)$ i.e., $S_1$ consists of traces of eighth powers of matrices in $M_n(R),$
read modulo $8R.$ Thus, $S_1=\{x_0^8+2x_1^4+4x_2^2\pmod{8R}\;:\;x_0,x_1,x_2\in R\}.$
It is easy to check (going modulo $8R$) that $S_1$ is closed under addition, using the fact that $W(2, 3, R)$ is
closed under addition already. Noting that $8s = 0,$ for all $s\in S_1,$ we get that $7s$ serves as the
additive inverse of an element $s \in S_1.$

\begin{theorem}\label{thm_tr8}
  Let $R$ be a commutative ring with unity and $n \geq 2$ be an integer.
    Then $M \in M_n(R)$ is a sum of eighth powers in $M_n(R)$ if and only if Tr $M \pmod{8R} \in S_1,$ where 
  $S_1=\{x_0^8+2x_1^4+4x_2^2\pmod{8R}: x_0,x_1,x_2\in R\}.$
\end{theorem}
\begin{proof}
  If $M \in M_n(R)$ is a sum of eighth powers of matrices, by Proposition \ref{prop_traceps} and using
  the fact that $S_1$ is closed under addition, we get that Tr $M \pmod{8R} \in S_1.$ 

  We now prove the converse. Note that by Remark \ref{rem_thm3.5} and using the same observation as in Corollary
  \ref{cor_equaltwoenough}, it is enough to prove the converse for $n = 2.$

  For this assume that $M\in M_2(R)$ is such that Tr $M \pmod{8R} \in S_1.$ Hence there
 exist $a,b,c,d \in R$ such that, 

 \begin{eqnarray} \mbox{Tr }M&=&a^8+2b^4+4c^2+8d \nonumber\\ &=& \mbox{Tr}\begin{pmatrix}	a & 0 \\ 0 & 0	\end{pmatrix}^8+ \mbox{Tr}\begin{pmatrix}	0 & b \\ -1 & 0	\end{pmatrix}^8+ 4c^2+8d \nonumber
	\end{eqnarray}
	Hence it suffices to prove that $4c^2$ and $8d$ are both sums of traces of eighth powers of matrices.

       \begin{itemize}
       \item[(a)]
         We first prove that $4c^2$ is a sum of traces of eighth powers of matrices. For this,
          we write down a few matrices  with the traces of their eighth powers. 
        \begin{table}[H]
   	\centering
	\begin{tabular}{|c|l|l|}\hline
			\multicolumn{2}{|c|}{Matrix $P$}& \hspace{2cm}$\mbox{         Trace } P^8            $  \\ \hline 
		$P_1$ & $\begin{pmatrix}	1 & -c \\ 1 & -1	\end{pmatrix}$ & $2c^4-8c^3+12c^2-8c+2$ \\ \hline
                $P_2$ & $\begin{pmatrix}	1 & -c^2-1 \\ 1 & 0	\end{pmatrix}$ & $2c^8-8c^6-16c^4-8c^2-1$ \\ \hline 
		$P_3$ & $\begin{pmatrix}	1 & -c+1 \\ -1 & 0	\end{pmatrix}$ & $2c^4+8c^3-16c^2+8c-1$ \\ \hline 
		$P_4$ & $\begin{pmatrix}	c-1 & c \\ 1 & 1	\end{pmatrix}$ & $c^8+8c^6+20c^4+16c^2+2$ \\ \hline
        \end{tabular}
        \end{table} \vspace{-1cm}
        Observe with $P_i$ as defined above, we have: 
	\begin{eqnarray}
	4c^2&=& 12c^2-8c^2 \nonumber \\
	&=& (2c^4-8c^3+12c^2-8c+2) +(2c^8-8c^6-16c^4-8c^2-1) \nonumber \\ && -2c^8+8c^6+14c^4+8c^3+8c-1 \nonumber \\ 
	&=& \sum_{i=1}^2 \mbox{Tr }P_i^8 -2c^8+8c^6+12c^4+(2c^4+8c^3-16c^2+8c-1)+16c^2 \nonumber \\
	&=& \sum_{i=1}^3 \mbox{Tr }P_i^8  -3c^8+(c^8+8c^6+20c^4+16c^2+2) -8c^4-2 \nonumber \\
	&=& \sum_{i=1}^4 \mbox{Tr }P_i^8  -23c^8+(1-8c^4+20c^8-16c^{12}+2c^{16})-3+16c^{12}-2c^{16} \nonumber \\
	&=& \sum_{i=1}^4 \mbox{Tr}P_i^8  +23\mbox{Tr}\begin{pmatrix}	c & c^2 \\ -1 & 0 \end{pmatrix}^8 +\mbox{Tr}\begin{pmatrix}	1 & c^4 \\ -1 & 0	\end{pmatrix}^8+3\mbox{Tr}\begin{pmatrix}	1 & 1 \\ -1 & 0	\end{pmatrix}^8 \nonumber \\ &&+8\mbox{Tr}\begin{pmatrix}	0 & c^3 \\ -1 & 0	\end{pmatrix}^8 +2\mbox{Tr}\begin{pmatrix}	c^2 & c^4 \\ -1 & 0	\end{pmatrix}^8.\nonumber 
	\end{eqnarray}

      \item[(b)] Now we will find matrices in $M_2(R)$, whose sum of the traces of eighth powers is $8d.$ Again we do this
        by writing a table of some useful matrices:

        \begin{tabular}{|c|l|l|}
          \hline
          	\multicolumn{2}{|c|}{Matrix $M$}& \hspace{2cm}$\mbox{         Trace $M^8$   }$  \\ \hline 
		$M_1$ & $\begin{pmatrix}	-d & d+1 \\ -1 & 0	\end{pmatrix}$ & $d^8-8d^7+12d^6+24d^5-26d^4-40d^3-4d^2+8d+2$
                \\ \hline
                $M_2$ & $\begin{pmatrix}	d+1 & -d^2 \\ 1 & 0	\end{pmatrix}$ & $-d^8+8d^7+12d^6-24d^5-30d^4+8d^3+20d^2+8d+1$
                \\ \hline 
		$M_3$ & $\begin{pmatrix}	d+1 & d+1 \\ -1 & 0	\end{pmatrix}$ & $d^8-8d^6-8d^5+12d^4+24d^3+12d^2-1$ \\ \hline 
		$M_4$ & $\begin{pmatrix}	-d+1 & d-1\\ 1 & 0	\end{pmatrix}$ & $d^8-8d^6+8d^5+12d^4-24d^3+12d^2-1$ \\ \hline
                $M_5$ & $\begin{pmatrix}	1 & -d\\ -1 & 0	\end{pmatrix}$ & $ 2d^4+16d^3+20d^2+8d+1$ \\ \hline
		$M_6$ & $\begin{pmatrix}	1 & -d-1\\ -1 & 0	\end{pmatrix}$ & $ 2d^4+24d^3+80d^2+104d+47$ \\ \hline
		$M_7$ & $\begin{pmatrix}	1 & -2d-2\\ 1 & 0	\end{pmatrix}$ & $ 32d^4-112d^2-112d-31$ \\ \hline
		$M_8$ & $\begin{pmatrix}	1 & d+1\\ -1 & 0	\end{pmatrix}$ & $ 2d^4-8d^3-16d^2-8d-1$ \\ \hline
		$M_9$ & $\begin{pmatrix}	1 & -d^2-1\\ 1 & 0	\end{pmatrix}$ & $ 2d^8-8d^6-16d^4-8d^2-1$ \\ \hline
		$M_{10}$ & $\begin{pmatrix}	1 & -d+1\\ -1 & 0	\end{pmatrix}$ & $2d^4+8d^3-16d^2+8d-1$ \\ \hline 
		$M_{11}$   & $\begin{pmatrix}	1 & -d\\ 1 & -1	\end{pmatrix}$ & $2d^4-8d^3+12d^2-8d+2$ \\ \hline
\end{tabular}

        With $M_i$ as above, we have: 
	\begin{eqnarray}
	  8d&=& (d^8-8d^7+12d^6+24d^5-26d^4-40d^3-4d^2+8d+2)\nonumber \\ && \ \ \ \ \ \ \
          -d^8+8d^7-12d^6-24d^5+26d^4+40d^3+4d^2-2 \nonumber \\
	&=& \mbox{Tr }M_1^8+(-d^8+8d^7+12d^6-24d^5-30d^4+8d^3+20d^2+8d+1) \nonumber \\ &&  -24d^6+56d^4+32d^3-16d^2-8d-3 
        \nonumber \\
	&=&  \sum_{i=1}^2 \mbox{Tr }M_i^8+ (d^8-8d^6-8d^5+12d^4+24d^3+12d^2-1) \nonumber \\ &&  -d^8-16d^6+8d^5+ 44d^4+8d^3-28d^2-8d-2 
        \nonumber \\ 
	&=&  \sum_{i=1}^3 \mbox{Tr }M_i^8+ (d^8-8d^6+8d^5+12d^4-24d^3+12d^2-1) \nonumber \\ &&  -2d^8-8d^6+32d^4+32d^3-40d^2-8d-1 
        \nonumber \\ 
	&=&  \sum_{i=1}^4 \mbox{Tr }M_i^8 -2d^8-8d^6+30d^4+(2d^4+16d^3+20d^2+8d+1) \nonumber \\ && + 16d^3-60d^2-16d-2,
        \nonumber \\  
	&=&  \sum_{i=1}^5 \mbox{Tr }M_i^8+ (2d^4+24d^3+80d^2+104d+47) + [(32d^4-112d^2) \nonumber \\
        && -(112d+31)] + (-2d^8-8d^6-4d^4-8d^3-28d^2-8d-18), 
        \nonumber \\ %
	&=&  \sum_{i=1}^7 \mbox{Tr }M_i^8 -2d^8-8d^6 -6d^4+(2d^4-8d^3-16d^2-8d-1)\nonumber \\ 
        && -12d^2-17, \nonumber \\ 
	&=&  \sum_{i=1}^8 \mbox{Tr }M_i^8 -4d^8+(2d^8-8d^6-16d^4-8d^2-1)+10d^4-4d^2-16,
        \nonumber \\  
	&=&  \sum_{i=1}^{9} \mbox{Tr }M_i^8 -4d^8+6d^4+(2d^4+8d^3-16d^2+8d-1)\nonumber \\
	&&+  (2d^4-8d^3+12d^2-8d+2)	-17,
        \nonumber \\
	&=&  \sum_{i=1}^{11} \mbox{Tr }M_i^8 -4d^8+6d^4 -17,
        \nonumber \\ 
	&=&  \sum_{i=1}^{11} \mbox{Tr }M_i^8 +4\mbox{Tr }\begin{pmatrix}	d & d^2\\ -1 & 0\end{pmatrix}^8 + 3\mbox{Tr }\begin{pmatrix}	0 & d\\ -1 & 0	\end{pmatrix}^8 +17\mbox{ Tr } \begin{pmatrix}	1 & 1\\ -1 & 0	\end{pmatrix}^8.\nonumber
	\end{eqnarray}

        \end{itemize} 
\end{proof} 

\section{Discriminant Criteria for a Matrix to be a Sum of Sixth and Eighth Powers of Matrices}
Let $K$ be an algebraic number field and $\mathcal{O}$ be an order of $K$. Here we will find the condition on the discriminant of the order $\mathcal{O}$ so that every matrix over $\mathcal{O}$ is a sum of sixth (eighth) powers of matrices. We will use following lemma: 
\begin{lemma} [\cite{ag}, Theorem 3.2] \label{le_power}
	Let $\mathcal{O}$ be an order in a number field $K$ and $p$ is a prime. The following are equivalent: 
	\begin{enumerate}
	\item Every element of $\mathcal{O}$ is a $p^{th}$ power modulo $p\mathcal{O}.$
        \item $x\in \mathcal{O}, x^p\in p\mathcal{O}$ imply $x\in p\mathcal{O}.$
		\item $(p,\mbox{disc } \mathcal{O})=1.$
	\end{enumerate}
\end{lemma}


\begin{theorem}
  Let $\mathcal{O}$ be an order of an algebraic number field and let $n \geq 2$ be an integer. 
  Then every matrix in $M_n(\mathcal{O})$ is a sum of sixth powers in $M_n(\mathcal{O})$ if and only if $(6, \text{ disc }\mathcal{O})=1$.
\end{theorem}
\begin{proof}
  If every $M \in M_n(\mathcal{O})$ is a sum of sixth powers of matrices, then clearly every matrix in $M_n(\mathcal{O})$ is a
  sum of squares and also a sum of cubes of matrices in $M_n(\mathcal{O}).$ Therefore by Theorem (\ref{th_dis_23457}), we get $(2, \text{ disc }\mathcal{O})= (3, \text{ disc }\mathcal{O}) = 1.$ Combining we have $(6, \text{ disc }\mathcal{O})= 1.$

  Now assume that $(6, \text{ disc }\mathcal{O})=1.$ Given $M \in M_n(\mathcal{O})$ we have Tr $M,$ say $\alpha$
  is in $\mathcal{O}.$ Since $(3, \mbox{ disc }\mathcal{O})=1$, by Lemma (\ref{le_power})
  $\alpha =x^3-3y$ for some $x,y \in \mathcal{O}.$ Since $(2, \mbox{disc }\mathcal{O})=1,$
  we have that $x=a^2-2b$ and $y=\alpha^2-2\beta$ for some $a,b,\alpha,\beta \in \mathcal{O}.$
  This gives $\alpha=(a^2-2b)^3-3(\alpha^2-2\beta) \equiv a^6-2b^3+3\alpha^2 (\mbox{mod }6\mathcal{O}).$
  Thus Tr $M = \alpha \pmod{6R} \in S$.
  Hence by Corollary \ref{cor_equaltwoenough}, $M$ is a sum of sixth powers of matrices. 
\end{proof}
\begin{theorem}
  Let $\mathcal{O}$ be an order in a number field. Then every matrix over $\mathcal{O}$ is a sum of eighth powers
  in $M_n(\mathcal{O})$ if and only if $(2,\mbox{disc } \mathcal{O})=1.$
\end{theorem}
\begin{proof}
  If every $M \in M_n(\mathcal{O})$ is a sum of eighth powers in $M_n(\mathcal{O}),$
  it is also a sum of squares in $M_n(\mathcal{O}).$ By Theorem (\ref{th_dis_23457}), we get that 
     $(2, \text{ disc }\mathcal{O}) = 1.$ 
  
  We now assume that $(2,\mbox{disc }\mathcal{O})=1$ and let $M\in M_n(\mathcal{O})$ be arbitrary.
  Then, $\alpha = \mbox{Tr }M \in \mathcal{O}$.
  As $(2,\mbox{ disc } \mathcal{O})=1$, each element of $\mathcal{O}$ is a square modulo $2R$. Hence there exist
  $l,m\in \mathcal{O}$ such that Tr $M=l^2+2m$. By the same argument, write $l=l_0^2+2l_1$ and $m=m_0^2+2m_1,$ to get  
		Tr $M=(l_0^2+2l_1)^2+2(m_0^2+2m_1) = l_0^4+2m_0^2+4(l_0^2l_1+l_1^2+m_1).$ 
	Let $n_0=l_0^2l_1+l_1^2+m_1 \in \mathcal{O}.$ 
	Then Tr $M=l_0^4+2m_0^4+4n_0.$ Again let $l_0=a_0^2+2a_1$ and $m_0=b_0^2+2b_1$, and $n_0=c_0^2+2c_1$, Then,
	Tr $M=(a_0^2+2a_1)^4+2(b_0^2+2b_1)^2+4(c_0^2+2c_1).$ Going modulo eight, 
        $\mbox{Tr }M=a_0^8+2b_0^4+4c_0^2\pmod{8R}.$ 
	Hence Tr $M \pmod{8R}$ is in $S_1$. Therefore by Theorem \eqref{thm_tr8}, $M$ is a sum of eighth powers
        in  $M_n(\mathcal{O}).$ 
\end{proof}

    {\bf Acknowledgment:} The first author would like to sincerely acknowledge the support that he got from his Principal and the motivation which he got from his colleagues Dr. Siby Abraham, Mrs. Latha Mohanan, Dr. Mithu Bhattacharya of Guru Nanak Khalsa College, Mumbai
    throughout the preparation of this paper. The authors would also like to thank Professor R. P. Deore, Head, Department of
    Mathematics, University of Mumbai for his continuous support and guidance.

\noindent {\bf~Corresponding author:} \\
Anuradha S. Garge,\\
Department of Mathematics, University of Mumbai, \\
Kalina Campus, Mumbai-400098, India. \\
Email: anuradha.garge@gmail.com

\begin{thebibliography}{99} 
	\bibitem{gk} Griffin M., Krusemeyer M., \textit{Matrices as  Sums of squares}, Linear and Multilinear Algebra {\bf 5}  (1977) 33 -- 44.
	\bibitem{lv} Vaserstein L. N., \textit{On the Sum of Powers of Matrices}, Linear and Multilinear Algebra {\bf 21}  (1987) 261 -- 270.
	\bibitem{mn} Newman M., {\it Sum of Squares of Matrices} Pacific Journal of Mathematics, {\bf 118} No. 2 (1985), 497--506.
	\bibitem{dr} Richman D. R., \textit{The Waring's Problem for Matrices}, Linear and Multilinear Algebra {\bf 22} no. 2 (1987) 171 - 192.
	\bibitem{kg} Katre S. A., Garge A. S., \textit{ Matrices over Commutative Ring  as Sums of $k$th Powers}, Proc. Amer. Math. Soc. {\bf 141}  no. 1 (2013), 103-113.
	\bibitem{ag} Garge A. S., \textit{Matrices over commutative rings as sums of fifth and seventh powers of matrices}, Linear and Multilinear Algebra (2019).
	\bibitem{kk} Katre S. A., Khule S. A., \textit{ Matrices over Order in  Algebraic Number Field as Sums of $k^{th}$ Powers}, Proc. Amer. Math. Soc. {\bf 128},  no. 3 (2000), 671 - 675.
        \bibitem{kshk} Wadikar Kshipra K, Katre S. A. \textit{Matrices over commutative ring with unity as sums of cubes}, Proceeding
          of International Conference on Engineering trends in mathematical and computational applications, $2010,$ Dec. $16-18,$
          Sivakasi, India: Allied Publishers, $2010$ p. $8-12.$ 
          \bibitem{kd} Katre S. A., Krishnamurthi Deepa., \textit{Matrices over non-commutative rings as sums of powers}, Linear and
          Multilinear Algebra, 2020. 

        \bibitem{la} Mc Laughlin J. \textit{Combinatorial identities deriving from the $n^{th}$ power of a $2 \times 2$ matrix}, Integers  no. 4 (2004).
\end{thebibliography}
\end{document}